\documentclass[10pt]{amsart}

\usepackage[a4paper,top=2cm,bottom=2cm,left=2cm,right=2cm]{geometry}

\usepackage{amsthm}
\usepackage{hyperref}
\usepackage{amsrefs}
\usepackage[mathscr]{eucal}

\theoremstyle{plain}
\newtheorem{theorem}{Theorem}[section]
\newtheorem{lemma}[theorem]{Lemma}

\newtheorem{proposition}[theorem]{Proposition}
\newtheorem*{claim*}{Claim}
\newtheorem{example}{Example}

\theoremstyle{definition}
\newtheorem{definition}[theorem]{Definition}
\newtheorem{question}[theorem]{Question}

\theoremstyle{remark}

\newenvironment{proof of claim}{\paragraph{\textbf{Proof of claim}}}{\hfill$\triangle$}

\numberwithin{equation}{section}

\begin{document}


\title[$G_\delta$-refinements]{$G_\delta$-refinements}

\author{Robson A. Figueiredo}
\address{Instituto de Matem\'atica e Estat\'istica da Universidade de S\~ao Paulo\\
Rua do Mat\~ao, 1010, Cidade Universit\'aria, CEP 05508-090, S\~ao Paulo, SP, Brazil}
\email{robs@ime.usp.br}



\begin{abstract}
In this work we deal with the preservation by $G_\delta$-refinements. 
We prove that for $\mathrm{SP}$-scattered spaces the metacompactness, paralindel\"ofness, metalindel\"ofness and linear lindel\"ofness are preserved by $G_\delta$-refinements. 
In this context we also consider some other generalizations of discrete spaces like $\omega$-scattered and $N$-scattered. 
In the final part of this paper we look at a question of Juh\'asz, Soukup, Szentmikl\'ossy and Weiss concerning the tighness of the $G_\delta$-refinement of a $\sigma$-product. 
\end{abstract}

\maketitle



\section{Preliminaries}

\begin{definition}
For any space $\langle X,\tau\rangle$, the topology $\tau_\delta$ obtained by letting every $G_\delta$ subset of $X$ be open is called the \textbf{$G_\delta$-topology} and the space so obtained is denoted by $X_\delta$.
\end{definition}

\begin{definition}
Let $X$ be a set and let $\mu$ be a cardinal such that $\mu\leq |X|$. 
We say that $\mathscr{C}\subseteq[X]^\mu$ is \textit{cofinal} in $[X]^\mu$ if, for all $x\in [X]^\mu$, there exists $y\in\mathscr{C}$ such that $x\subseteq y$.
For cardinals $\mu\leq\kappa$, we define $\mathrm{cf}\left([\kappa]^\mu,\subseteq\right)$ as the least cardinality of a cofinal family in $[\kappa]^\mu$.
Given a infinite cardinal $\kappa$, we define $\mathrm{Cov}_\omega(\kappa)=\mathrm{cf}\left([\kappa]^{\aleph_0},\subseteq\right)$.
\end{definition}

\begin{theorem}[Passos \cite{Pas2007}]\label{1}
Let $\kappa$ be a infinite cardinal such that $\mathrm{Cov}_\omega(\kappa)=\kappa$.
Given a set $X$ of cardinality $\kappa$, there exists an $\omega$-covering elementary submodel $M$ such that $X\subseteq M$ and $|M|=\kappa$.
\end{theorem}

Recall that a subset $F$ of $X$ is $\kappa$-closed, where $\kappa$ is an infinite cardinal iff whenever $S\subseteq F$ and $|S|\leq\kappa$ then $S\subseteq F$.
It is well known that $t(X)\leq\kappa$ iff every $\kappa$-closed set in $X$ is closed.

\section{ \texorpdfstring{$\mathrm{SP}$}{P}-scattered spaces}

\begin{definition}
A point $p$ in a topological space $X$ is called a \textit{strong $P$-point} if it has a neighborhood consisting of $P$-points. 
The set of all strong $P$-points of $X$ is denoted by $\mathrm{SP}(X)$.
\end{definition}

Observe that $\mathrm{SP}(X) = \mathrm{int}_XP(X)$.

\begin{definition}[\cite{HRW2007}]
Recursively, define:
\begin{itemize}
\item $S_0(X) = X$ and $S_1(X) = X\setminus\mathrm{SP}(X)$;
\item $S_{\alpha+1}(X) = S_1(S_\alpha(X))$ for any ordinal $\alpha\geq 1$;
\item $S_\lambda(X) =\bigcap\{\,S_\alpha(X):\alpha<\lambda\,\}$, if $\lambda$ is a limit ordinal.
\end{itemize}
\end{definition}

In \cite{HRW2007}, Henriksen, Raphael and Woods proved the following generalizations of the well known theorems 5.1 and 5.2 of \cite{LR1981}, respectively:

\begin{theorem}
If $X$ is a Lindel\"of $\mathrm{SP}$-scattered regular space, then $X_\delta$ is Lindel\"of.
\end{theorem}

\begin{theorem}\label{HRW2007:4.2}
If $X$ is a paracompact $\mathrm{SP}$-scattered Hausdorff space, then $X_\delta$ is paracompact.
\end{theorem}

In the same article they asked: 

\begin{question}
If $X$ is a metacompact  $\mathrm{SP}$-scattered regular space, so is $X_\delta$ a metacompact space?
\end{question}

In this section, we will see that not just the metacompactness, but also the paralindel\"ofness, the metalindel\"ofness and the linear lindel\"ofness are preserved by $G_\delta$-refinements on the class of $\mathrm{SP}$-scattered regular spaces.

\begin{theorem}[\cite{Mis1972}]\label{142}
A regular $P$-space $X$ is paralindel\"of if, and only if, it is paracompact.
\end{theorem}

\begin{proposition}[\cite{HRW2007}]\label{HRW2007:2.8}
If $X$ is a regular space, then the following are equivalent:
\begin{enumerate}
\item $X$ is $\mathrm{SP}$-scattered;
\item if $A\subseteq X$ is nonempty, then $\mathrm{int}_A\{\,a:\text{$a$ is a $P$-point of $A$}\,\}\ne\emptyset$.
\end{enumerate}
\end{proposition}

\begin{theorem}\label{2}
If $X$ is a regular $\mathrm{SP}$-scattered paralindel\"of space, then $X_\delta$ is paracompact. 
\end{theorem}
\begin{proof}
By the theorem \ref{142}, it is enough to show that $X_\delta$ is paralindel\"of.
Let $\mathscr{C}$ be an open cover of $X_\delta$.
Let $O$ be the set of all points $x\in X$ such that $x\in \mathrm{int}_\tau\bigcup\mathscr{C}'$ for some locally countable open partial refinement $\mathscr{C}'$ of $\mathscr{C}$ in $X_\delta$.

If $O=X$ then, for each $x\in X$, there exists a locally countable open partial refinement $\mathscr{C}_x$ of $\mathscr{C}$ in $X_\delta$ such that $x\in V_x=\mathrm{int}_\tau\bigcup\mathscr{C}_x$.
Since $\langle X,\tau\rangle$ is paralindel\"of, the open cover $\mathscr{V}=\{\,V_x:x\in X\,\}$ admits a locally countable open refinement $\{\,W_s:s\in S\,\}$, with $W_s\ne W_{s'}$ whenever $s\ne s'$.
For each $s\in S$, take $x_s\in X$ such that $W_s\subseteq V_{x_s}$.
So 
\[
\{\,W_s\cap C:s\in S\text{ and }C\in\mathscr{C}_{x_s}\,\}
\]
is locally countable open refinement of $\mathscr{C}$ in $X_\delta$.

Now it remains to show that in fact $O=X$.
Suppose this is not the case.
Since $\langle X,\tau\rangle$ is $\mathrm{SP}$-scattered, by the proposition \ref{HRW2007:2.8}, there exists a $y\in X\setminus O$ and an open neighborhood $U$ of $y$ in $X$ such that $(X\setminus O)\cap U$ is a $P$-subspace of $X$.
Take $C_y\in\mathscr{C}$ such that $y\in C_y$.
We can suppose that $C_y=\bigcap\{\,U_n:n\in\omega\,\}$, where, for each $n\in\omega$, $U_n\in \tau$ and $\mathrm{cl}_\tau(U_n)\subseteq U$.
So $C_y\cup O$ is an open subset of $X$, for $(X\setminus O)\cap C_y$ is a subspace of the $P$-space $U$.
As $X$ is regular, $y$ has an open neighborhood $U_y$ in $X$ such that $\mathrm{cl}_{\tau}(U_y)\subseteq C_y\cup O$.

Fix $n\in\omega$.
Let $F=\mathrm{cl}_\tau(U_y)\setminus U_n$.
Since $F\subseteq O$, for each $x\in F$, there exists a locally countable open partial refinement $\mathscr{C}_x$ of $\mathscr{C}$ in $X_\delta$ such that $x\in V_x=\mathrm{int}_\tau\bigcup\mathscr{C}_x$.
As $\langle X,\tau\rangle$ is paralindel\"of and $F$ is closed, $\mathscr{V}=\{\,V_x:x\in F\,\}$ admits a locally countable open partial refinement $\mathscr{W}=\{\,W_s:s\in S\,\}$ that covers $F$, where $W_s\ne W_{s'}$ whenever $s\ne s'$.
For each $s\in S$, choose $x_s\in F$ such that $W_s\subseteq V_{x_s}$.
Consider the family
\[
\mathscr{D}_n=\{\,W_s\cap C:s\in S\text{ e }C\in\mathscr{C}_{x_s}\,\}.
\]

\textit{The family $\mathscr{D}_n$ is a locally countable open cover of $F$ in  $X_\delta$.}
Indeed, let $x\in X$.
Since $\mathscr{W}$ is a locally countable open family in $X$, there exists an open neighborhood $Z_x$ of $x$ in $X$ such that $T=\{\,s\in S:W_s\cap Z_x\ne\emptyset\,\}$ is countable.
For each $t\in T$, take an open neighborhood $O_t$ of $x$ in $X_\delta$ such that $\{\,C\in\mathscr{C}_{x_t}:C\cap O_t\ne\emptyset\,\}$ is countable.
Consider the following open neighborhood of $x$ em $X_\delta$:
\[
Z=Z_x\cap\bigcap\{\,O_t:t\in T\,\}.
\]
Note that for each $t\in T$, $\mathscr{C}_{x_t}(Z)=\{\,C\in\mathscr{C}_{x_t}:C\cap Z\ne\emptyset\,\}$ is countable.
Seeing that
\begin{align*}
\mathscr{D}_n({Z})&=\{\,D\in\mathscr{D}_n:D\cap Z\ne\emptyset\,\}\\
            &=\{\,W_i\cap C: s\in S,\,C\in\mathscr{C}_{x_s}\text{ e }W_s\cap C\cap Z\ne\emptyset\,\}\\
             &\subseteq\{\,W_t\cap C: t\in T\text{ e }C\in\mathscr{C}_{x_t}({Z})\,\},
\end{align*}
the family $\mathscr{D}_n({Z})$ is countable.

Thus,
\[
\mathscr{C}'=\{C_y\}\cup\bigcup\{\,\mathscr{D}_n:n\in\omega\,\}
\]
is a locally countable open partial refinement of $\mathscr{C}$ in $X_\delta$ such that $y\in U_y\subseteq\mathrm{int}_\tau\bigcup\mathscr{C}'$, contradicting the fact that $y\notin O$.
\end{proof}

\begin{theorem}
If $X$ is a regular $\mathrm{SP}$-scattered metalindel\"of space, then $X_\delta$ is metalindel\"of. 
\end{theorem}
\begin{proof}
Let $\mathscr{C}$ be an open cover of $X_\delta$ and consider the set $O$ whose elements are all those $x\in X$ such that $x\in\mathrm{int}_\tau\bigcup\mathscr{C}'$ for some pointwise countable open partial refinement $\mathscr{C}'$ of $\mathscr{C}$ in $X_\delta$.

If $O=X$ then, proceeding in the same way as in the proof of theorem \ref{2}, we can obtain a pointwise countable open refinement of $\mathscr{C}$ in $X_\delta$.

In order to complete the proof, it is enough to show that $O=X$.
Suppose on the contrary that $O\ne X$.
As $X$ is regular and $\mathrm{SP}$-scattered, by the proposition \ref{HRW2007:2.8}, there exist a point $y\in X\setminus O$ and an open neighborhood $U$ of $y$ in $X$ such that $U\cap (X\setminus O)$ is a $P$-subspace of $X$.
Choose a $C_y\in\mathscr{C}$ such that $y\in C_y$.
We can suppose that $C_y=\bigcap\{\,U_n:n\in\omega\,\}$, where for each $n\in\omega$, $U_n\in\tau$ and $\mathrm{cl}_\tau(U_n)\subseteq U$.
Note that $C_y\cup O$ is an open subset of $X$.
Hence, from the regularity of $X$ it follows that there exist an open neighborhood $U_y$ of $y$ in $X$ such that $\mathrm{cl}_\tau(U_y)\subseteq C_y\cup O$. 

Fix $n\in\omega$.
Note that $F=\mathrm{cl}_\tau(U_y)\setminus U_n\subseteq O$.
Then, for each $x\in F$, there is a pointwise countable open partial refinement $\mathscr{C}_x$ of $\mathscr{C}$ in $X_\delta$ such that $x\in V_x=\mathrm{int}_\tau(\bigcup\mathscr{C}_x)$.
Because $X$ is metalindel\"of and $F$ is closed then $\{\,V_x:x\in F\,\}$ has a pointwise countable open partial refinement $\mathcal{W}=\{\,W_i:i\in I\,\}$, where $W_i\ne W_j$ whenever $i\ne j$.
For each $i\in I$, choose $x_i\in F$ such that $W_i\subseteq V_{x_i}$.
Consider the family
\[
\mathscr{D}_n=\{\,W_i\cap C:i\in I\text{ e }C\in\mathscr{C}_{x_i}\,\}.
\]

It is easily checked that each $\mathscr{D}_n$ is a pointwise countable open partial refinement of $\mathscr{C}$ which covers $\mathrm{cl}_{\tau}(U_y)\setminus U_n$.
So,
\[
\mathscr{C}'=\{C_y\}\cup\bigcup\{\,\mathscr{D}_n:n\in\omega\,\}
\]
is a pointwise countable open partial refinement of $\mathscr{C}$ in $X_\delta$ such that $y\in U_y\subseteq\mathrm{int}_\tau(\bigcup\mathscr{C}')$, contradicting the fact that $y\notin O$.
Thus, $O=X$.
\end{proof}

\begin{theorem}
If $X$ is a regular $\mathrm{SP}$-scattered metacompact space, then $X_\delta$ is metacompact. 
\end{theorem}
\begin{proof}
Let $\mathscr{C}$ be an open cover of $X_\delta$ and consider the set $O$ whose elements are all $x\in X$ such that $x\in\mathrm{int}_\tau(\bigcup\mathscr{C}')$ for some pointwise finite open partial refinement $\mathscr{C}'$ of $\mathscr{C}$ in $X_\delta$.

Similarly to what it has been done in theorem \ref{2}, we can get, from the assumption $O=X$, a pointwise finite open refinement of $\mathscr{C}$ in $X_\delta$.

We complete the proof by showing that $O=X$.
Suppose that $O\ne X$.
Since $\langle X,\tau\rangle$ is $\mathrm{SP}$-scattered and regular, by the proposition \ref{HRW2007:2.8}, there are a point $y\in X\setminus O$ and an open neighborhood $U$ of $y$ in $X$ such that $(X\setminus O)\cap U$ is a $P$-subspace of $X$.
Take a $C_y\in\mathscr{C}$ such that $y\in C_y$.
We can suppose that $C_y=\bigcap\{\,U_n:n\in\omega\,\}$, where for each $n\in\omega$, $U_n\in\tau$ and $\mathrm{cl}_\tau(U_{n+1})\subseteq U_n\subseteq\mathrm{cl}_\tau(U_n)\subseteq U$.
Note that $C_y\cup O$ is an open subset of $X$.
Once $X$ is regular, $y$ has an open neighborhood $U_y$ in $X$ such that $\mathrm{cl}_\tau(U_y)\subseteq(C_y\cup O)\cap U_0$.

As $F_n\subseteq O$, for each $x\in F_n$, there exists a pointwise finite open partial refinement $\mathscr{C}_x$ of $\mathscr{C}$ in $X_\delta$ such that $x\in V_x=\mathrm{int}_\tau(\bigcup\mathscr{C}_x)$.
Since $\langle X,\tau\rangle$ is metacompact and $F_n$ is closed then $\mathcal{V}=\{\,V_x:x\in F_n\,\}$ has a pointwise finite open partial refinement $\mathcal{W}$ which covers $F_n$.
We can suppose that $\mathcal{W}=\{\,W_i:i\in I\,\}$, where $W_i\ne W_j$ whenever $i\ne j$.
For each $i\in I$, choose $x_i\in F_n$ such that $W_i\subseteq V_{x_i}$.
Consider the family
\[
\mathscr{D}_n=\{\,W_i\cap C\cap(U_n\setminus\mathrm{cl}_\tau(U_{n+2}):i\in I\text{ and }C\in\mathscr{C}_{x_i}\,\}.
\]
Note that each $\mathscr{D}_n$ is a pointwise finite open partial refinement of $\mathscr{C}$ in $X_\delta$ such that $\bigcup\mathscr{D}_n=U_n\setminus \mathrm{cl}_\tau(U_{n+2})$.
Therefore,
\[
\mathscr{C}'=\{C_y\}\cup\bigcup\{\,\mathscr{D}_n:n\in\omega\,\}
\]
is a pointwise finite open partial refinement of $\mathscr{C}$ in $X_\delta$ such that $y\in U_y\subseteq\mathrm{int}_\tau(\bigcup\mathscr{C}')$, contradicting the fact that $y\notin O$.
\end{proof}

\begin{theorem}
If $X$ is a regular $\mathrm{SP}$-scattered linearly Lindel\"of space, then $X_\delta$ is linearly Lindel\"of. 
\end{theorem}
\begin{proof}
Let $\mathscr{C}=\{\,C_\alpha:\alpha<\kappa\,\}$ be an open cover of $X_\delta$, where $\kappa$ is an uncountable regular cardinal.
For each $\alpha<\kappa$, let
\[
V_\alpha=\mathrm{int}_\tau\left(\bigcup\{\,C_\beta:\beta\leq\alpha\,\}\right).
\]
Define
\[
O=\left\{\,x\in X:\text{ there exists $\alpha(x)<\kappa$ such that $x\in V_{\alpha(x)}$}\,\right\}.
\]

\begin{claim*}
$O=X$ 
\end{claim*}
\begin{proof of claim}
Suppose on the contrary that $O\ne X$.
As $\langle X,\tau\rangle$ is a $\mathrm{SP}$-scattered regular space, by the proposition \ref{HRW2007:2.8}, there are $y\in X\setminus O$ and an open neighborhood $U$ of $y$ in $X$ such that $(X\setminus O)\cap U$ is a $P$-subspace of $X$.
Choose $\alpha_y<\kappa$ such that $y\in C_{\alpha_y}$.
We can suppose that $C_{\alpha_y}=\bigcap\{\,U_n:n\in\omega\,\}$, where, for each $n\in\omega$, $U_n\in\tau$ and $\mathrm{cl}_\tau(U_n)\subseteq U$.
Note that $C_y\cup O$ is an open subset of $X$.
Once $X$ is regular, $y$ has an open neighborhood $U_y$ in $X$ such that $\mathrm{cl}_\tau(U_y)\subseteq C_{\alpha_y}\cup O$.

Fix $n\in\omega$.
Let $F_n=\mathrm{cl}_\tau(U_y)\setminus U_n$.
Note that $F_n$ is a closed subset of $X$ and so it is linearly Lindel\"of.
Moreover, $F_n\subseteq O$;
this implies that, for each $x\in F_n$, we can take $\alpha(x)<\kappa$ such that $x\in V_{\alpha(x)}$.
Then, $\left\{\,V_{\alpha(x)}:x\in F_n\,\right\}$ is a family of open subsets of $X$ which covers $F_n$ and it is linearly ordered by inclusion.
Therefore, there is a countable subset $E_n\subseteq F_n$ such that $\left\{\,V_{\alpha(x)}:x\in E_n\,\right\}$ covers $F_n$.
So, $\alpha=\sup(\{\alpha_y\}\cup\bigcup\{\,E_n:n\in\omega\,\})<\kappa$ e
\[
y\in U_y\subseteq C_{\alpha_y}\cup\bigcup\{\,F_n:n\in\omega\,\}\subseteq\bigcup\{\,C_\beta:\beta\leq\alpha\,\}.
\]
Then $y\in V_\alpha$ and, thus, $y\in O$.
This is a contradiction.
\end{proof of claim}

By the claim above, $\mathcal{V}=\{\,V_\alpha:\alpha<\kappa\,\}$ is an open cover of $\langle X,\tau\rangle$.
Since $\langle X,\tau\rangle$ is a linearly Lindelöf space, $\mathcal{V}$ has a subcover whose cardinality is less than $\kappa$.
Because $\kappa$ is regular, $V_\alpha=X$ for some $\alpha<\kappa$.
Thus, $\{\,C_\beta:\beta<\alpha\,\}$ is a subcover of $\mathscr{C}$ whose cardinality is less than $\kappa$.
\end{proof}

\section{Other generalizations of scattered}

Clearly, if a regular Lindel\"of space $X$ is a countable union of scattered closed subspaces, then $X_\delta$ is Lindel\"of. As we shall see, at least consistently, this is not the case when it is not required that the subspaces are closed.

A space is \textbf{$\sigma$-scattered} if it is an union of a countable family of scattered subspaces. 

\begin{example}\label{3}
Assuming $\mathsf{CH}$, there exists a regular $\sigma$-scattered Lindel\"of space whose $G_\delta$-refinement is not Lindel\"of.
\end{example}
\begin{proof}
It is enough to take a Luzin subset of the real line containing the rational numbers and consider it as a subspace of the Michael line. 
\end{proof}

\begin{question}
Is there a regular $\sigma$-scattered Lindel\"of space whose $G_\delta$-refinement is not Lindel\"of?
\end{question}

\begin{question}
Is there a regular $\sigma$-scattered paracompact space whose $G_\delta$-refinement is not paracompact?
\end{question}

Hdeib and Pareek introduced in \cite{HP1989} the following natural generalization of scattered spaces: a space $X$ is \textbf{$\omega$-scattered} if, for each non-empty subset $A$ of $X$, there exist a point $x\in A$ and an open neighborhood $U_x$ of $x$ such that $U_x\cap A$ is countable.

Every scattered space is $\omega$-scattered, but the reverse is not true: the set of rational numbers with the usual topology is $\omega$-scattered and non-scattered.

The theorem 3.12 of \cite{HP1989} states that in the class of regular $\omega$-scattered spaces the Lindel\"of property is preserved by $G_\delta$-refinements. 
However, this is not true once the space of the example \ref{3} is $\omega$-scattered.

\begin{question}
Is there a Hausdorff $\omega$-scattered paracompact space $X$ such that $X_\delta$ is not paracompact?
\end{question}

A space $X$ is \textbf{$N$-scattered} if every nowhere dense subset of $X$ is a scattered subspace of $X$. 
The next example was noticed by Santi Spadaro.

\begin{example}
Assuming $\mathsf{CH}$, there exists a $N$-scattered Lindel\"of space whose $G_\delta$-refinement is not Lindel\"of.
\end{example}
\begin{proof}
Let $\mathcal{M}$ the family of all Lebesgue measurable subsets of the real line.
For each $E\in\mathcal{M}$, define
\[
\Phi(E)=\left\{\,x\in\mathbb{R}:\lim_{h\to 0}\frac{m\left(E\cap{]x-h,x+h[}\right)}{2h}=1\,\right\}.
\] 
Then
\[
\tau_d=\left\{\,E\in\mathcal{M}:E\subseteq\Phi(E)\,\right\}
\]
is a topology on $\mathbb{R}$ stronger than that usual, well known as \textit{density topology}. 
Denote by $\mathbb{R}_d$ the topological space $\langle\mathbb{R},\tau_d\rangle$.
By corollary 4.3 of \cite{Tal1976}, $\mathsf{CH}$ implies that  $\langle\mathbb{R},\tau_d\rangle$ has a hereditarily Lindel\"of, non-separable, regular and Baire subspace $X$.
By theorem 2.7 of \cite{Tal1976}, every nowhere dense subset of $X$ is discrete (and closed). 
Therefore, $X$ is $N$-scattered.
On the other side, the pseudocharacter of $X$ is countable, for $X$ is Hausdorff and hereditarily Lindel\"of.
Then $X_\delta$ is discrete and uncountable and, thus, it is not Lindel\"of.
\end{proof}

\begin{question}
Is there a Hausdorff paracompact $N$-scattered space whose $G_\delta$-refinement is not a paracompact space?
\end{question}

\section{The tightness of  \texorpdfstring{$G_\delta$}{Gdelta}-refinement of  \texorpdfstring{$\sigma$}{sigma}-products}

Given a  family of topological spaces $\{\,X_i:i\in I\,\}$ and a point $x^\ast\in X=\prod\{\,X_i:{i\in I}\,\}$, define
\[
\sigma=\sigma(X)=\sigma(X,x^\ast)=\left\{\,x\in \prod_{i\in I}X_i:\mathrm{supp}(x)\text{ is finite}\,\right\},
\]
where $\mathrm{supp}(x)=\{\,i\in I:x(i)\ne x^\ast(i)\,\}$.
The \textit{$\sigma$-product of $X$ at $x^\ast$} is the set $\sigma$ equipped with the topology induced by the Tychonoff product $\prod\{\,X_i:{i\in I}\,\}$\@.

In \cite{JSSW}, Juh\'asz, Soukup, Szentmikl\'ossy and Weiss proved:

\begin{theorem}
Let $\kappa$ and $\lambda$ be cardinals, with $\kappa\leq\aleph_1$.
Let $X$ be the one point lindel\"ofication of a discrete space of cardinality $\kappa$ by a point $p$ and let $x^\ast\in X^\kappa$, where $x^\ast(\alpha)=p$ for all $\alpha<\kappa$.
Then $\left(\sigma(X^\kappa,x^\ast)\right)_\delta$ has tighness $\aleph_1$.
\end{theorem}

In the same article, it was asked:

\begin{question}
Assume that $X$ is a Lindel\"of $P$-space such that $t(X)=\aleph_1$.
Is it true that 
\[t(\sigma(X^\kappa)_\delta)=\aleph_1\]
for all cardinal $\kappa$? 
\end{question}

We will see that the answer is positive.


%
\begin{lemma}[\cite{DW1986}]\label{160}
If $\{\,X_i:i\leq n\,\}$ is a finite family of regular locally Lindel\"of $P$-spaces, then  
\[
t\left(\prod\{\,X_i:i\leq n\,\}\right)=\max\{\,t(X_i):i\leq n\,\}.
\]
\end{lemma}

\begin{lemma}\label{4}
If $X=\sigma\{\,X_n:n\in\omega\,\}$ is a $\sigma$-product of regular locally Lindel\"of $P$-spaces, then
\[
t\left(X_\delta\right)=\sup\{\,t(X_n):n\in\omega\,\}.
\] 
\end{lemma}
\begin{proof}
Let $\lambda=\sup\{\,t(X_n):n\in\omega\,\}$.
Let $Y$ be a non-closed subset of $\sigma_\delta=\sigma\left(X,x^\ast\right)_\delta$ and let $q\in \mathrm{cl}(Y)\setminus Y$.
For each $n\in\omega$, let
\[
Y_n=\{\,y\in Y:\mathrm{supp}(y)\subseteq n\,\}.
\]
Since $Y=\bigcup\{\,Y_n:n\in\omega\,\}$ and $\sigma(X,x^\ast)_\delta$ is a $P$-space, there exists a $m\in\omega$ such that $q\in\mathrm{cl}(Y_m)$.
Now, $\pi_m(q)\in\mathrm{cl}\left(\pi_{m}[Y_m]\right)$, where $\pi_m$ is the natural projection from $\prod\{\,X_n:n\in\omega\,\}$ in $\prod\{\,X_i:i\in m\,\}$.
Since by the corollary \ref{160} the tightness of $\prod\{\,X_i:i\in m\,\}$ is $\leq\lambda$, there exists $Z'\subseteq \pi_{m}[Y_m]$ of cardinality $\leq\lambda$ such that $\pi_m(q)\in\mathrm{cl}(Z')$.
Then $Z=Z'\times\prod\{\,\{x^\ast(n)\}:n\geq m\,\}\subseteq Y_m\subseteq Y$ and since $\mathrm{supp}(q)\subseteq m$, we have $q\in\mathrm{cl}(Z)$.
\end{proof}

\begin{lemma}\label{161}
Let $\kappa$ be a infinite cardinal.
Let $X=\prod\{\,X_\alpha:\alpha<\kappa\,\}$. 
If for each countable subset $I\subseteq \kappa$, $\sigma\{\,X_\alpha:\alpha\in I\,\}_\delta$
has tightness $\leq\lambda$, with $\mathrm{Cov}_\omega(\lambda)=\lambda$, then $\sigma(X)_\delta$ has tightness $\leq\lambda$.
\end{lemma}
\begin{proof}
Let $Y$ be a non-closed subset of $\sigma_\delta=\sigma(X,x^\ast)_\delta$ and let $q\in\mathrm{cl}(Y)\setminus Y$.
By the theorem \ref{1}, there exists a $\omega$-covering elementary submodel $M$ of cardinality $\lambda$ such that $\{X,x^\ast,\kappa,\lambda,q,Y\}\cup\lambda\subseteq M$.
We are going to show that $q\in\mathrm{cl}(Y\cap M)$.
Suppose that 
\[
q\in U=\prod\{\,U_\alpha:\alpha\in I\,\}\times \prod\{\,X_\alpha:\alpha\in\kappa\setminus I\,\},
\]
where $I$ is a countable subset of $\kappa$ and each $U_\alpha$ is an open subset of $X_\alpha$\@.
Since $M$ is $\omega$-covering, there exists $J\in M$, a countable subset of $\kappa$ such that $I\cap M\subseteq J$.
Now note that $\pi_J(q)\in \mathrm{cl}(\pi_J[Y])$ and $\pi_J[Y]\in M$; besides $\pi_J[\sigma(X,x^\ast)_\delta]$ belongs to $M$ and, 
by hypothesis, its tightness is $\leq\lambda$.
So by the elementarity there exists $Z\in M$, a subset of $\pi_J[Y]$ whose cardinality is at most $\lambda$, such that $\pi_J(q)\in\mathrm{cl}(Z)$.
Let $z\in \pi_J[U]\cap Z$.
Note that $z\in M$, because, since $Z\in M$ and $Z$ has cardinality at most $\lambda$ and $\lambda\subseteq M$, $Z\subseteq M$.
So by the elementarity there exists $y\in Y\cap M$ such that $\pi_J(y)=z$.
	
We claim that $y\in U$.
Indeed, since $\mathrm{supp}(y)$, $\mathrm{supp}(q)\subseteq M$, it follows that if $\alpha\in I\setminus M$ so $y(\alpha)=x^\ast(\alpha)=q(\alpha)\in U_\alpha$\@.
On the other side, if $\alpha\in I\cap M$ so $\alpha\in J$ and thus $y(\alpha)=z(\alpha)\in U_\alpha\cap M$\@.
Therefore, $y\in \prod\{\,U_\alpha:{\alpha\in I}\,\}\times\prod\{\,X_\alpha:\alpha\in\kappa\setminus I\,\}$\@.
\end{proof}

%
%
%
\begin{theorem}
If $X=\prod\{\,X_\alpha:\alpha<\kappa\,\}$, with each $X_\alpha$ being a Lindel\"of $P$-space such that $t(X_\alpha)\leq\lambda$, with $\mathrm{Cov}_\omega(\lambda)=\lambda$, then 
\[
t\left(\sigma\left(X\right)_\delta\right)\leq \lambda.
\]
In particular, if $X$ is a Lindel\"of $P$-space whose tightness is $\aleph_n$ then the tightness of $\sigma\left(X^\kappa, x^\ast\right)_\delta$ is $\aleph_n$.
\end{theorem}

As a corollary of the previous theorem we have that, for a regular Lindel\"of $P$-space,
\[
t\left(\sigma\left(X^\kappa\right)_\delta\right)\leq t(X)^{\aleph_0}.
\]

It remains to be seen whether:

\begin{question}
Is there a Lindel\"of $P$-space $X$  such that $t(X)=\lambda$, with $\mathrm{Cov}_\omega(\lambda)>\lambda$, and
\[
t\left(\sigma\left(X^\kappa\right)_\delta\right)>\lambda?
\]
\end{question}

\begin{question}
Assuming that $\mathrm{Cov}_\omega(\aleph_\omega)=\aleph_{\omega+1}$, is there a Lindel\"of $P$-space $X$  such that $t(X)=\aleph_\omega$ and
\[
t\left(\sigma\left(X^\kappa\right)_\delta\right)=\aleph_{\omega+1}?
\]
\end{question}

Based on the theorem 3.1 from \cite{DW1986}, we have:

\begin{lemma}\label{6}
Let $\lambda$ be a infinite cardinal.
If $X=\sigma\{\,X_\alpha:\alpha<\lambda\,\}$ is a $\sigma$-product of regular Lindel\"of $P$-spaces, then 
\[
t\left(X_\delta\right)\leq \mathrm{Cov}_\omega(\lambda)\cdot \sup\{\,t(X_\alpha):\alpha<\lambda\,\}.
\]
\end{lemma}
\begin{proof}
Let $\kappa=\mathrm{Cov}_\omega(\lambda)\cdot \sup\{\,t(X_\alpha):\alpha<\lambda\,\}$.
For each $I\subseteq \lambda$, let $\sigma_I=\sigma\{\,X_\alpha:\alpha\in I\,\}_\delta$.
Suppose that $A\subseteq \sigma_\delta$ is $\kappa$-closed, and $a\in\mathrm{cl}_{\sigma_\delta}(A)$.

Note that, for each countable subset $J\subseteq \lambda$, $\pi_J[\mathrm{cl}_{\sigma(\lambda)}(A)]\subseteq \mathrm{cl}_{\sigma_J}(\pi_J[A])$.
Indeed, let $x\in \pi_J[\mathrm{cl}_{\sigma(\lambda)}(A)]$. 
Then there exists $z\in \mathrm{cl}_{\sigma(\lambda)}(A)$ such that $\pi_J(z)=x$. 
If $\prod\{\,U_j:j\in J\,\}$ is an basic neighborhood of $x$ in $\sigma_J$, then $\prod\{\,U_j:j\in J\,\}\times\prod\{\,X_\alpha:\alpha\in{\lambda\setminus J}\,\}$ is an open neighborhood of $z$. 
So $\left(\prod\{\,U_j:j\in J\,\}\times \prod\{\,X_\alpha:\alpha\in{\lambda\setminus J}\,\}\right)\cap A\ne \emptyset$ and, thus, $\left(\prod\{\,U_j:j\in J\,\}\right)\cap\pi_J[A]\ne \emptyset$.
Therefore, $x\in \mathrm{cl}_{\sigma_J}(\pi_J[A])$.

Then, for each countable subset $J\subseteq \lambda$, $\pi_J(a)\in \mathrm{cl}_{\sigma_J}(\pi_J[A])$.
By lemma \ref{4}, $t(\sigma_J)\leq\kappa$, then we can take $B_J\subseteq \pi_J[A]$ of cardinality $\leq \kappa$ such that $\pi_J(a)\in\mathrm{cl}_{\sigma_J}(B_J)$. 
For each $b\in B_J$ choose $x_b\in A$ such that $\pi_J(x_b)=b$, and let $C_J=\{\,x_b:b\in B_J\,\}$.

Now, let $\mathcal{J}$ be a cofinal family in $[\lambda]^{\aleph_0}$ and let
\[C=\bigcup\{\,C_J:J\in\mathcal{J}\,\}.\]
Note that $|C|\leq \mathrm{Cov}_\omega(\lambda)\cdot t(X)=\kappa$.
Then $\mathrm{cl}_{\sigma(\lambda)}(C)\subseteq A$.
So, it remains to be proved that $a\in \mathrm{cl}_{\sigma(\lambda)}(C)$.
Let $U=\prod\{\,U_j:{j\in J'}\,\}\times\prod\{\,X_\alpha:\alpha\in{\lambda\setminus J'}\,\}$ be an basic neighborhood of $a$ in $\sigma(\lambda)$.
Let $J\in\mathcal{J}$ such that $J'\subseteq J$.
Since $\pi_J(a)\in \mathrm{cl}_{\sigma_J}(B_J)=\mathrm{cl}_{\sigma_J}\left(\pi_J[C_J]\right)$, then $\pi_J[U]\cap\pi_J[C_J]\ne\emptyset$; so $U\cap C_J\ne\emptyset$.
Therefore, $U\cap C\ne\emptyset$.
\end{proof}

In the same way we have proved the lemma \ref{4}, we can show the following result for the cases in which $\mathrm{Cov}_\omega(t(X))>t(X)$:

\begin{lemma}\label{5}
If $X$ is a Lindel\"of $P$-space then 
\[
t(\sigma(X^{\aleph_\omega},x^\ast)_\delta)=t(X).
\] 
\end{lemma}
\begin{proof}
Let $\kappa=t(X)$.
Suppose that $\mathrm{Cov}_\omega(\kappa)>\kappa$.
Note that $\kappa\geq\aleph_\omega$.
Let $Y$ be a non-closed subset of $\sigma(X^\omega,x^\ast)_\delta$ and let $q\in \mathrm{cl}(Y)\setminus Y$.
For each $n\in\omega$, let
\[
Y_n=\{\,y\in Y:\mathrm{supp}(y)\subseteq \omega_n\,\}.
\]
Since $Y=\bigcup\{\,Y_n:n\in\omega\,\}$ and $\sigma(X^\omega,x^\ast)_\delta$ is a $P$-space, there exists a $m\in\omega$ such that $q\in\mathrm{cl}(Y_m)$.
Now, $\pi_m(q)\in\mathrm{cl}\left(\pi_{m}[Y_m]\right)$, where $\pi_m$ is the natural projection from $X^{\aleph_\omega}$ in $X^{\aleph_m}$.
Since by the theorem \ref{6} the tightness of $X^{\aleph_m}$ is $\kappa$, there exists $Z'\subseteq \pi_{m}[Y_m]$ of cardinality $\leq\kappa$ such that $\pi_m(q)\in\mathrm{cl}(Z')$.
Then $Z=Z'\times\prod\{\,\{x^\ast(n)\}:n\geq m\,\}\subseteq Y_m\subseteq Y$ and since $\mathrm{supp}(q)\subseteq \omega_m$, we have $q\in\mathrm{cl}(Z)$.
\end{proof}






%

\end{document}